%% file: compactG.tex
\documentclass{article}
\usepackage{appendix}
\usepackage{amsmath, amsfonts, amssymb, amsthm, graphicx} 
\usepackage{color,xcolor}
\usepackage{caption}
\usepackage[bookmarks]{hyperref}
\usepackage{hyperref}
\hypersetup{
    colorlinks=true,   	
    linkcolor=red,      
    citecolor = [rgb]{0 0.7 0},   	
    filecolor=magenta, 	
    urlcolor=blue
}

\input{paper_defns.tex}

\newcommand{\dstar}{d^{*}_{\mathrm{R}}}
\newcommand{\dR}{d_{\mathrm{R}}}
\newcommand{\IS}{\mathbb{I}_S(X)}
\newcommand{\kk}{\kappa}

\title{Entropy power inequalities in compact groups}

\author{Lampros Gavalakis%
	\thanks{Statistical Laboratory, DPMMS,
	University of Cambridge,
	Centre for Mathematical Sciences,
        Wilberforce Road,
	Cambridge CB3 0WB, U.K.
        Email: {\tt lg560@cam.ac.uk}.
	}
\and 
Ioannis Kontoyiannis%
\thanks{Statistical Laboratory, DPMMS,
	University of Cambridge,
	Centre for Mathematical Sciences,
        Wilberforce Road,
	Cambridge CB3 0WB, U.K.
        Email: {\tt yiannis@maths.cam.ac.uk}.
	}
\and 
Sharang M. Sriramu%
\thanks{School of Electrical and Computer Engineering,
Cornell University,
Ithaca, NY, USA.
Email: {\tt sms579@cornell.edu}.
}
\and 
Aaron Wagner
\thanks{School of Electrical and Computer Engineering,
Cornell University,
Ithaca, NY, USA.
Email: {\tt wagner@cornell.edu}.
}
}

\date{\today}

\begin{document}

\maketitle

\begingroup
\renewcommand\thefootnote{}
\footnotetext{L.G.\ and I.K.\ were supported in part
	by the EPSRC funded INFORMED-AI project EP/Y028732/1.}
\endgroup

\begin{abstract}
Suppose $X,Y$ are independent random variables with values 
in a compact abelian group $(G,+)$. We examine the following two
entropy power-type inequalities:
$h(X+Y)\geq \frac{1}{2}h(X)+\frac{1}{2}h(Y)$
and $h(X+Y)\geq \max\{h(X),h(Y)\}$,
where the entropy $h(Z)$ of a $G$-valued random variable $Z$ 
is defined in terms of its density with respect 
to Haar measure on $G$.
For groups that are either connected or finite with 
no nontrivial subgroups, 
we precisely characterize the cases
of equality 
and establish explicit,
quantitative stability 
estimates 
in terms of relative entropy
for these two inequalities.
The main tools are a generalization of an entropic inequality obtained 
by Green, Manners and Tao (2023) for discrete entropy,
and a harmonic-analytic estimate for the chi-squared 
contraction coefficient in connected compact groups. 
As an application, 
we derive exponential convergence rates to the uniform distribution 
in relative entropy for random walks on 
connected compact abelian groups.
\end{abstract}

\section{Introduction}
\subsection{Shannon's continuous Entropy Power Inequality in $\RL$}

One of the formulations 
of Shannon's Entropy Power Inequality (EPI) over the real 
numbers~\cite{shannon:48,stam:59,blachman:65}
states that, for any pair of independent
random variables $X,Y$ in $\mathbb{R}$,
\begin{equation} 
\label{EPIR}
h(X+Y) \geq \frac12 h(X) + \frac12h(Y) + \frac{1}{2}\log{2},
\end{equation}
with equality if and only if
$X$ and $Y$ are Gaussian with the same 
variance~\cite{EPI:25arxiv}.
Here, $h(X)$ denotes the differential entropy of $X$,
defined as
$$h(X):=-\int_{\RL} f(x)\log f(x)\,dx,$$
whenever the law of $X$ has a density $f$ with respect to Lebesgue 
measure and the integral exists, and $h(X)=-\infty$ otherwise.
[Throughout, `$\log$' denotes the natural logarithm.]

The question of stability in the EPI~\eqref{EPIR}
has been studied extensively, often in connection 
with convex geometry. Among numerous related results, 
Eldan and Mikulincer~\cite{eldan:20} identified appropriate
conditions under which
near-equality in~\eqref{EPIR} implies that
$X$ and $Y$ are close to being Gaussian.
Specifically, they obtained {\em quantitative} stability
bounds stating, e.g., that if $X,Y$ are log-concave with
unit variance then
\be
D(X\|Z)+D(Y\|Z)\leq
k \,C(X,Y)^3
\Big[h(X+Y)-\frac{1}{2}h(X)-\frac{1}{2}h(Y)-\frac{1}{2}\log 2\Big].
\label{eq:EM}
\ee
Here, $k$ is an explicit universal constant,
$C(X,Y)$ is the largest of the Poincar\'{e} constants
of $X$ and $Y$, 
$D(X\|Y)$ denotes the relative entropy
between the laws of two random variables $X,Y$,
and $Z$ is a standard normal.
Recall that, for any two probability measures $\mu,\nu$
on the same space, the relative entropy between
them is 
$$
D(\mu\|\nu):=
\begin{cases} \int \log\big(\frac{d\mu}{d\nu}\big)\,d\mu& \mbox{if}\;\mu\ll\nu,\\
+\infty & \mbox{otherwise}.
\end{cases}
$$
As we discuss below, 
statements like~\eqref{eq:EM} are also referred to as
{\em entropy jumps}
in the literature.

In the special case when $X,Y$ are i.i.d.,
Ball, Barthe and Naor~\cite{ball:03} and
Kontoyiannis and Madiman~\cite{KM:14} 
have obtained more general stability bounds,
assuming only that $X$ has
a finite Poincar{\'e} constant.
These were generalized to higher dimensions under the extra 
assumption that $X$ is log-concave by Ball and Nguyen~\cite{ball:12}, 
and to non-i.i.d.\ log-concave random vectors 
by Eldan and Mikulincer~\cite{eldan:20}. 
It is also known that stability in the sense of relative entropy may 
fail without assumptions~\cite{courtade:18}. For a more extensive
discussion of the relevant literature
see~\cite{EPI:25arxiv}, 
where a qualitative stability result
for the EPI~\eqref{EPIR} was also shown 
under much weaker conditions.

\subsection{Discrete EPIs in abelian groups}

Let $H(X)$ denote the usual (discrete) Shannon 
entropy of a discrete random variable $X$ taking values in some set $A$,
$$H(X):=-\sum_{x\in A}\BBP(X=x)\log\BBP(X=x).$$
As in the continuous case~\eqref{EPIR},
discrete EPIs generally provide 
lower bounds for $H(X+Y)$, when $X$ and $Y$ are 
independent discrete random variables
with values in some ambient group $G$.
A typical starting point is the elementary inequality
\be
H(X+Y)\geq H(X),
\label{eq:elem}
\ee
and sharper bounds are then established 
for specific groups $G$.

Indeed, various such discrete generalizations 
of~\eqref{EPIR} have been proposed.
A discrete analogue of the EPI in finite groups was obtained 
in~\cite{jog:14} through the point of view of minimizing 
the entropy of a sum under fixed entropies of the summands. 
This extends the binary result of~\cite{shamai:90}
that was based on what is widely known as Mrs.\ Gerber's 
Lemma~\cite{wyner:73I,wyner:73II}.
Another discrete approach 
gives EPI-like results for uniformly distributed
random variables with values in finite subsets
of $\IN$~\cite{madiman:21b}. More recently,
lower bounds for $H(X+Y)$ in connection with the
Cauchy-Davenport theorem were established in~\cite{gavalakis-goh:arxiv}.

One of the discrete EPI-like bounds closest in spirit to this work
is the following result of Tao~\cite{tao:10}: For i.i.d.\ discrete 
random variables 
$X,X'$ taking
values in a torsion-free group, the simple bound~\eqref{eq:elem} 
can be strengthened to
\be
H(X+X') \geq H(X) + \frac12\log{2} -o(1),
\label{eq:tao}
\ee
as $H(X)\to \infty$. 
For this bound,
a discrete analogue of the stability result~\eqref{eq:EM} 
was recently obtained in~\cite{EPI:25arxiv} for log-concave 
random variables using convexity arguments, where the discrete
Gaussian plays the role of the (asymptotic) extremizer in~(\ref{eq:tao}).

In a similar spirit, when $X$ and $Y$ are independent $\IN$-valued
random variables, Haghighatshoar, Abbe and Telatar~\cite{abbe:14} 
obtained a stability result for the elementary
bound
$$H(X+Y)\geq\frac{1}{2}H(X)+\frac{1}{2}H(Y),$$
showing that 
there exists a continuous function $g:\mathbb R_+^2 \to \mathbb R_+$ 
with $g(x) = 0$ if and only if $x=0$ such that 
\begin{equation}\label{HATeq}
    H(X+Y)-\frac12H(X)-\frac12H(Y) \geq g(H(X),H(Y)).
\end{equation}

\subsection{Main results: General EPIs on compact groups} 
\label{introresults}

In this work we consider sums of independent random variables
$X,Y$ with values in a compact group $G$, and we obtain 
sharp lower bounds
for the entropy of their sum $h(X+Y)$, where now the entropy
is defined in terms of Haar measure on $G$. 
To be precise, let $(G,+)$ be a compact and metrizable 
topological abelian group, and let $\mu$ denote the unique
Haar probability measure on $(G,\clG)$, where $\clG$ is the 
Borel $\sigma$-algebra.
The metrizability assumption makes $G$ a Polish space,
which guarantees the existence of regular conditional 
distributions and allows us to avoid measurability problems. 
When the law of a $G$-valued random variable $X$ has a density 
$f$ with respect to $\mu$,
we define its entropy as
\be
h(X):=-\int_G f(x)\log f(x)\,d\mu(x)
=-\int_G f(x)\log f(x)\,dx,
\label{eq:hHaar}
\ee
and set $h(X)=-\infty$ otherwise. Since $\mu$ is a probability 
measure, the integral in the definition always exists
and $h(X)\in[-\infty,0]$.
If $G$ is a discrete group, 
$h(X) \in [-\log{M(G)}, 0],$
where $M$ is the counting measure on $G$.
At the lower extreme,
$h(X)=-\log M(G)=\log\mu(\{0\})$ if and only if $X$ is deterministic,
and at the other end,
$h(X) = 0$ if and only if 
$X$ is uniform on the whole group $G$,
i.e., if $X$ has law $\mu$. 

For any two independent random variables in $G$, the basic inequality
\be
h(X+Y)\geq\frac{1}{2}h(X)+\frac{1}{2}h(Y),
\label{eq:EPIG}
\ee
follows easily from the stronger bound
\be
h(X+Y)\geq\max\big\{h(X),h(Y)\big\},
\label{eq:mEPIG}
\ee
which is a simple consequence of Jensen's inequality.
We refer to~\eqref{eq:EPIG} as ``the EPI on $G$''
and to~\eqref{eq:mEPIG} as ``the maximal EPI on $G$''.

After carefully characterizing the 
cases of equality in~\eqref{eq:EPIG}
and in~\eqref{eq:mEPIG},
the main goal of this work is
to establish explicit stability
results for these two EPIs;
see Theorems~\ref{stabilityTh} and~\ref{maxentropyepi}
below.

For the EPI~\eqref{eq:EPIG},
it is not very hard to see that there is equality if and only if $X,Y$ 
are both uniformly distributed on an open subgroup, at least under 
reasonably mild 
assumptions. If $G$ has proper open subgroups, then there is no 
hope of having stability results in relative entropy.
For example, take
\(G=\mathbb F_2^2\), \(H=\{(0,0),(1,1)\}\),
and let $X$ be uniform on $H$ with probability $1-\varepsilon$ and 
equal to $(0,1)$ with probability $\varepsilon$. The distribution of $X$ 
is not absolutely continuous with respect to the uniform on any proper 
subgroup and the relative entropy between $X$ and the uniform on 
the whole group is bounded away from $0$ for any $\varepsilon$, 
while clearly $h(X+X')-h(X) \to 0$ as $\varepsilon\to 0.$

Therefore, we will restrict attention to the case where $G$ has
no open subgroups except possibly $\{0\}$. In this setting, if $\{0\}$ 
is open, then $G$ is discrete and finite.
Otherwise, if $\{0\}$ is not 
open, then $G$ is connected; see Lemma~\ref{lemmaproper}
in Section~\ref{preliminariessec}.
The connected compact
abelian setting has been considered in the context of Kneser's 
theorem~\cite{tao:18},  as well as in the context of Riesz--Sobolev 
inequalities~\cite{christ:22}; see the discussion at the 
end of Section~\ref{preliminariessec}.

Our first result,
proved in Section~\ref{proof2sec},
precisely characterizes the equality case in
the EPI~\eqref{eq:EPIG}.

\begin{theorem}[EPI on compact groups] \label{equalityTh}
Let $G$ be a compact and metrizable topological abelian group, equipped with its Borel $\sigma$-algebra. Assume that $G$ has no open subgroups except $G$ itself and possibly $\{0\}$. Let $X,Y$ be independent random variables on $G$ with finite entropies. Then
$$h(X+Y)-\frac{1}{2}h(X)-\frac{1}{2}h(Y) \geq 0.$$
Moreover:
\begin{itemize}
\item[$(i)$]
If $G$ is connected, then there is equality if and only 
if $X$ and $Y$ are both uniform on $G$.
\item[$(ii)$]
If $G$ is finite, then there is equality if and only if 
either both $X$ and $Y$ are uniform on the whole group $G$ 
or both $X$ and $Y$ are deterministic.
\end{itemize}
\end{theorem}

Our next result is a stability estimate for Theorem~\ref{equalityTh}. 

\begin{theorem}[Stability of EPI on compact groups] 
\label{stabilityTh}
Let $G$ be a compact and metrizable topological abelian group, equipped 
with its Borel $\sigma$-algebra. Assume that $G$ has no open subgroups 
except $G$ itself and possibly $\{0\}$. Let $X,Y$ be independent random 
variables on $G$ with finite entropies, and write:
$$\delta_{\mathrm{EPI}}:=h(X+Y)-\frac{1}{2}h(X)-\frac{1}{2}h(Y).$$
There are absolute constants $C_1, C_2,C_3 > 0$ such that the following hold:
   
   If $G$ is connected and $\delta_{\mathrm{EPI}} \leq C_1,$ then
        \begin{equation} \label{eq:stability}
        \max\big\{D(X\|U_{G}), D(Y\|U_{G})\big\}
\leq C_2  \sqrt{\delta_{\mathrm{EPI}}}
	\log\Big(\frac{1}{\delta_{\mathrm{EPI}}}\Big),
        \end{equation}
        where $U_{G}$ is the uniform on the whole group $G$. 

        If $G$ is finite and $\delta_{\mathrm{EPI}} \leq C_1,$ then 
either~\eqref{eq:stability} holds or 
        \begin{equation*}
        \log \mu(\{0\})\leq \min\{h(X),h(Y)\} \leq \max\{h(X),h(Y)\} \leq \log \mu(\{0\}) + C_3  \sqrt{\delta_{\mathrm{EPI}}}
\log\Big(\frac{1}{\delta_{\mathrm{EPI}}}\Big).
         \end{equation*}
\end{theorem}

In the 
recent work~\cite{blachar:arxiv}, structural characterizations are obtained 
for random variables with $\delta_{\mathrm{EPI}} = O(1)$ in locally compact 
groups. Those results hold in greater generality, but are only valid
up to $O(1)$ constants, whereas in Theorem~\ref{stabilityTh} we obtain 
$o(1)$-estimates in relative entropy as $\delta_{\mathrm{EPI}} \to 0$, 
which one cannot hope to get in more general settings. 

We prove Theorem~\ref{stabilityTh} in Section~\ref{proof2sec}. The main 
tool,
Proposition~\ref{revisitedprop}, is a generalization 
of~\cite[Proposition~1.2]{green:25} to compact groups.
If the group $G$ is connected, which is equivalent it not 
having proper open subgroups, we obtain an explicit strictly 
positive lower bound on the EPI-deficit in Theorem~\ref{12jump}. 
Note that Theorem~\ref{stabilityTh} already implies an analogue 
of~\eqref{HATeq} in finite groups with no open proper subgroups 
(such as $\mathbb F_p$ with $p$ prime). One may get an explicit 
function analogous to $g$ in~\eqref{HATeq} in a similar manner 
as in the proof of Theorem~\ref{12jump}.

Next, we turn to the maximal EPI in~\eqref{eq:mEPIG}.
In this case the conditions for 
equality case are not as clean as before. 
Suppose, for example, $h(Y) > h(X)$. There is clearly equality 
if $Y$ is uniform on $G$ and also if $X$ is deterministic, 
but there are also nontrivial equality cases where 
$h(X)=-\infty$ while $X$ is not deterministic. For example, 
let $Y$ have density $f_Y(y)=1+\frac12\cos{(4\pi y)}$
for $y$ in $G=\mathbb T = \mathbb R / \mathbb Z$,
and let $X$ be $0$ or $\frac12$, both with probability $\frac12$. 
Then by the symmetry of $f_Y$, $Y+X$ has the same distribution as $Y$. 

\begin{corollary}[Maximal EPI on connected compact groups]
\label{cor:mEPI}
Let $G$ be a compact and metrizable topological abelian group,
equipped with its Borel $\sigma$-algebra. For any pair
of independent random 
variables $X,Y$ on $G$,
$$h(X+Y)\geq \max\big\{h(X),h(Y)\big\}.$$
Moreover, suppose that $G$ is connected and 
$\max{\{h(X),h(Y)\}} = h(Y) >-\infty$. Then we have
equality if and only if either $Y$ is uniformly distributed
on $G$ or $h(X)=-\infty$.
\end{corollary}

The inequality in Corollary~\ref{cor:mEPI} is elementary,
and the characterization of the equality case 
actually follows from immediately from our next result,
Theorem~\ref{maxentropyepi}, which provides a stability
bound for the maximal EPI. This can also be viewed
as an ``entropy jump'' result or,
equivalently, it can be phrased as 
a {\em strong data processing inequality}; see, 
e.g.,~\cite{polyanskiy-wu:17}, and the references therein. 
Using a standard 
expression for the relative entropy as an integral of chi-squared 
divergences,
we reduce the problem to obtaining $L^2$ contraction estimates under 
convolution. Connectedness implies that the image of nontrivial 
Fourier coefficients is the whole group, and thus we are able to get 
a good $L^2$ estimate by explicitly bounding these coefficients. 
Theorem~\ref{maxentropyepi} is proved 
in Section~\ref{connectedSec} along these lines.

\begin{theorem}[Stability of the maximal EPI on compact groups]
\label{maxentropyepi}
Let $G$ be a connected, compact and metrizable topological abelian group, 
equipped with its Borel $\sigma$-algebra. Let $X,Y$ be independent random 
variables on $G$ such that $\max{\{h(X),h(Y)\}} = h(Y) >-\infty.$ 
Put 
$$s := e^{6h(X)-2} \quad \text{and} \quad \eta_s  
:= \left(\frac{\sin(\pi s)}{\pi s}\right)^2.$$ 
Then
$$
 h(X+Y) - h(Y) \geq \min{\Bigl\{\frac{1}{4}, 
\frac{(1-\eta_s)^2}{8+\frac{16\eta_s}{s}} \Bigr\}}D(Y\|U_{G}).
$$
\end{theorem}

We note that compactness is essential for Theorem~\ref{maxentropyepi}
and Corollary~\ref{cor:mEPI}.
For example, on
$\mathbb{R}$ there are no nontrivial equality cases in
the maximal EPI~\eqref{eq:mEPIG}. 

Finally, we observe that
Theorem~\ref{maxentropyepi} also implies exponential 
convergence rates for random walks in compact groups,
in the sense of relative entropy. We state this as
Corollary~\ref{rates} in 
Section~\ref{connectedSec}. This is a well studied topic; 
see~\cite{diaconis:book,major:79,rosenthal:94,varju:13}.
In particular, see
\cite{johnson:00b, harremoes:09b} for relative entropy convergence, 
\cite{madiman:21b} for discrete entropy inequalities in cyclic groups,
and the recent work~\cite{blachar:arxiv} for entropy analogues of Gromov's 
theorem on finitely generated groups. To illustrate our result in 
this setting, we note that for a sum 
$S_n=\sum_{i=1}^nX_i$
of i.i.d.\ random variables $X_i$ on $G$ 
with finite entropies, 
Corollary~\ref{rates} implies that,
\begin{equation*} 
D(S_{n+1}\|U_G) \leq r^n D(X_1\|U_G),
\end{equation*}
where $U_G$ is uniform on $G$ and $r \in (0,1)$ is a constant 
that depends only on the entropy of $X_1$.

\section{Lower bounds on the Ruzsa distance}
\subsection{Preliminaries} 
\label{preliminariessec}

Throughout, we take $(G,+)$ to be a compact, metrizable,
topological abelian group, equipped with its Borel $\sigma$-algebra
$\clG$, and with its unique Haar probability
measure $\mu$.
In all of our results we assume that $G$ has no 
open subgroups except $G$ itself and possibly $\{0\}$,
and in some cases we will also assume that $\{0\}$ is not open 
to get sharper bounds. 
The latter is equivalent to the group being connected,
as seen by the following lemma, which we will frequently use;
cf.~\cite[Corollary~7.9]{hewitt:bookI}.

\begin{lemma}
\label{lemmaproper}
For a locally compact group $G$ the following are equivalent: 
\begin{enumerate}
    \item $G$ is connected.
    \item $G$ has no proper open subgroups. 
\end{enumerate}
\end{lemma}

For any two (typically finite) subsets $S,S'$ of $G$, we
write $S+S'=\{s+s':s\in S, s'\in S'\}$ and 
$S-S'=\{s-s':s\in S, s'\in S'\}$ for their corresponding
sumset and difference set, respectively.

The entropy $h(X)$ of a random variable with values in $G$
is defined as in~\eqref{eq:hHaar}, and the joint entropy
$h(X_1,\ldots,X_n)$ of a $G^n$-valued random vector $(X_1,\dots,X_n)$ is 
similarly defined, with $\mu^n$ in place of $\mu$.
For two jointly distributed random variables $X,Y$,
the conditional entropy is defined as $h(X|Y) 
= \int h(X|Y=y)f_Y(y)\,dy$, where the last expression is valid 
whenever $Y$ has density $f_Y$ with respect to $\mu$,
as long as the integral is well defined.
Alternatively, 
$h(X|Y) = h(X,Y)-h(Y)$, as long as this expression
is not $-\infty-(-\infty)$.
If $X,Y$ are random variables with joint law $\mu_{X,Y}$
and marginal laws $\mu_X$ and $\mu_Y$, respectively,
the {\em mutual information} between $X$ and $Y$ is 
$$
I(X;Y) := D(\mu_{X,Y}\|\mu_X\times\mu_Y).
$$
When $X,Y$ take values in the groups $G,G'$ with associated
Haar probability measures $\mu,\mu'$,
and if their joint law has a density with respect to 
$\mu\times\mu'$, then
the mutual information can be expressed as
$$
I(X;Y) = h(X) - h(X|Y) = h(Y)-h(Y|X) \geq 0. 
$$
We write $X\sim {\rm Uni}(G)$ when $X$ is uniformly distributed 
on $G$, i.e., when $X$ has law $\mu$. We denote
the support of a random variable $X$ by ${\rm supp}(X)$,
and whenever a random variable, $X$, say, has a density 
with respect to $\mu$, we denote it by $f_X$.

We record some standard properties of the entropy,
relative entropy and mutual information.

\begin{lemma}[Data processing inequality~\cite{pinsker:book}]
Let $(S_1,\mathcal{S}_1), (S_2,\mathcal{S}_2)$ be measurable spaces 
and let $X,Y$ be random variables with values in $S_1$ and $S_2$ 
respectively. Let $(T_1,\mathcal{T}_1)$ be some other measurable 
space and $g:S_1 \to T_1$ 
be a measurable map. Then 
$$
I(X;Y) \geq I(g(X);Y).
$$
\end{lemma}

In Polish spaces, it is a consequence of the Donsker--Varadhan 
representation~\cite{donsker-varadhan:I-II}
that relative entropy is lower 
semicontinuous with respect to weak convergence:  

\begin{lemma}[Lower semicontinuity of relative entropy]
Let $\mu, \nu$ be probability measures on some measurable Polish space 
$(S_1,\mathcal{S}_1)$. Let $\{\mu_n\},\{\nu_n\}$ be sequences of measures 
such that $\mu_n\to\mu,\nu_n\to\nu$ weakly as $n\to \infty$. Then 
    $$
    \liminf_{n\to \infty}D(\mu_n\|\nu_n) \geq D(\mu\|\nu).
    $$
\end{lemma}

\begin{lemma} \label{lemmacontinuity}
    Let $G$ be a compact and metrizable topological abelian group, 
equipped with its Borel $\sigma$-algebra. Let $X$ be a random variable 
and $\{U_{\epsilon}\}_{\epsilon>0}$ be an indexed family
 of random variables on $G$, independent of $X,$ such that 
$U_{\epsilon} \to 0$ weakly as $\epsilon \to 0$. Then
    $$
    \lim_{\epsilon \to 0}h(X+U_\epsilon) = h(X).
    $$
\end{lemma}

\noindent
{\sc Proof.}
    We always have 
    $$
    h(X+U_\epsilon) \geq h(X+U_{\epsilon}|U_{\epsilon}) = h(X).
    $$

    Letting $U_{G}$ be uniform on $G$, we have
    $$
    h(X) = -D(X\|U_{G}).
    $$
    Since $G$ is metrizable and compact, it is also Polish. Thus, lower semicontinuity of relative entropy with respect to weak convergence yields 
    $$
    \limsup_{\epsilon\to 0}h(X+U_{\epsilon}) = -\liminf_{\epsilon \to 0}D(X+U_{\epsilon}\|U_{G})\leq -D(X\|U_{G})= h(X),
    $$
    which completes the proof. 
\qed       

We note that Lemma~\ref{lemmacontinuity} is valid even when
$h(X) = -\infty$.

As in the case of discrete Shannon entropy~\cite{green:25}, it 
is convenient to work with the Ruzsa distance: 

\begin{definition}
    Let $G$ be a compact and metrizable topological abelian group, equipped with its Borel $\sigma$-algebra. Let $X,Y$ be $G$-valued random variables with $h(X),h(Y) > -\infty$. 
    The {\em Ruzsa distance} between $X$ and $Y$ is 
    $$
    \dR(X,Y) =  h(X'-Y') - \frac{1}{2}h(X)-\frac{1}{2}h(Y),
    $$
    where $X',Y'$ are independent and have the same marginals as $X$ and $Y$,
	respectively.
    
    The {\em maximal Ruzsa distance} is
    $$
    \dstar(X,Y) = \sup h(X'-Y') - \frac{1}{2}h(X)-\frac{1}{2}h(Y)
    $$
    where the supremum is over all couplings $(X',Y')$ of $(X,Y)$, i.e.,
all jointly distributed random variables $(X',Y')$
with same marginals as $X$ and $Y$.
\end{definition}

The following proposition generalizes the discrete-entropy
result of~\cite[Lemma~1.1]{green:25} to the present setting, 
but it requires a different proof method. 

\begin{proposition} \label{startriangle}
Let $G$ be a compact and metrizable topological abelian group, equipped with its Borel $\sigma$-algebra. Assume $X,Y,Z$ are $G$-valued random variables with finite entropies. Then
    $$
    \dR(X,Z) \leq \dstar(X,Z) \leq \dR(X,Y)+\dR(Y,Z).
    $$
\end{proposition}

\noindent
{\sc Proof.}
    The first inequality is obvious by the definitions. 
    
    For the second, it suffices to show that, if $Y$ is independent of $(X,Z)$ (with $X$ and $Z$ not necessarily independent) and $h(X-Z) > -\infty$ (since this is the case for at least the independent coupling of $(X,Z)$), then 
    \begin{equation}\label{toprovetr}
        h(X-Z) \leq h(X-Y) + h(Y-Z) -h(Y). 
    \end{equation}
We proceed similarly to the proof of~\cite[Theorem~3.1]{KM:14}. 
By the data processing property of mutual information we always have
\begin{equation} \label{dp}
    I(X;(X-Y,Y-Z)) \geq I(X;X-Z).
\end{equation}

We will first prove~\eqref{toprovetr}, assuming that all the involved quantities are finite, so that the left-hand side in~\eqref{dp} satisfies the identity
\begin{equation} \label{firstwithfinite}
    I(X;(X-Y,Y-Z)) = h(X) + h(X-Y,Y-Z) - h(X,X-Y,Y-Z),
\end{equation}
and the right-hand side satisfies 
\begin{equation} \label{secondwithfinite}
    I(X;X-Z) = h(X-Z) - h(Z|X).
\end{equation}
In this case, since the map $(x,y) \mapsto (x,x-y)$ is one-to-one, we have
\begin{align*}
h(Y-Z) - h(Y-Z|X,X-Y) &= I(Y-Z;(X,X-Y)) \\
&= I(Y-Z;(X,Y))\\
&= h(Y-Z) - h(Y-Z|X,Y),
\end{align*}
and thus 
$$
h(Y-Z|X,X-Y)=h(Y-Z|X,Y).
$$
Therefore, 
\begin{align*}
h(X,X-Y,Y-Z) &= h(X)+h(X-Y|X)+h(Y-Z|X,X-Y)\\
& = h(X) + h(Y)+h(Y-Z|X,Y) \\
&= h(X,Y) +h(Z|X,Y)\\
& = h(X,Y,Z)\\
& = h(X,Z) + h(Y),
\end{align*}
where we used the chain rule for entropy twice, translation invariance,
and the independence of $Y$ and $(X,Z)$.

Putting these together and rearranging~\eqref{dp} gives 
$$
h(X) + h(X-Y,Y-Z) - h(X,Z) - h(Y) \geq h(X-Z) - h(Z|X).
$$
Now~\eqref{toprovetr} follows by $h(X-Y,Y-Z) \leq h(X-Y)+h(Y-Z)$ and the chain rule for $h(X,Z)$.

In the general case, we may assume that $\{0\}$ is not open as otherwise the 
group is discrete and finite and thus entropy is at least 
$\log{\mu(\{0\})}$ (and in particular finite). Then, 
apply~\eqref{toprovetr} to $\tilde{X} = X + U_{\epsilon}, \tilde{Y} 
= Y + U'_{\epsilon},\tilde{Z} = Z + U''_{\epsilon}$, 
where $U_{\epsilon},U_{\epsilon}',U_{\epsilon}''$ are independent 
uniforms on a neighborhood of $0$ of Haar measure at most $\epsilon$. 
Such a neighborhood exists by outer regularity of the Haar measure. 
It is then easy to check that all the quantities 
in~\eqref{firstwithfinite} and~\eqref{secondwithfinite} are finite. 
Taking $\epsilon \to 0$ and using Lemma~\ref{lemmacontinuity} 
completes the proof. 
\qed       

The sum-submodularity inequality for 
discrete Shannon entropy~\cite{KV:83} also holds on $G$~\cite{KM:14,KM:16}: 

\begin{proposition} \label{submodularity}
Let $G$ be an abelian, compact topological group, equipped with its Borel $\sigma$-algebra. Let $X,Y,Z$ be independent $G$-random variables. Then
\be
    h(X+Y+Z) +h(Y) \leq h(X+Y)+h(Y+Z).
\label{eq:sumsub}
\ee 
\end{proposition}

\noindent
{\sc Proof.}
If $h(Y) = -\infty$, there is nothing to prove. So we may assume that $h(Y) >-\infty$ and, in particular, that $Y$ has a density. Note then that all four
terms in~\eqref{eq:sumsub} are finite. 

Rearranging, the desired inequality is 
$$
I(X+Y+Z;X) \leq I(X+Y;X),
$$
but this follows from the data processing inequality 
for mutual information and independence. 
\qed

The following result is due to 
Kneser~\cite[Satz~1]{kneser:56}. It is worth noting that Kneser's 
theorem unifies the Brunn--Minkowski inequality in dimension one and 
the Cauchy--Davenport inequality. For us, it will be used as an 
intermediate step to show that a set $S$ for which $\mu(S-S)/\mu(S)$ 
is close to $1$ has either almost full or almost zero measure. An inverse 
theorem to Kneser's was obtained by Tao~\cite{tao:18} in 
connected compact abelian groups, showing that near-equality
in the lower abound 
implies that the sets are close to Bohr sets. A related functional 
inequality was studied in~\cite{christ:22}.

\begin{lemma}[Kneser]
Let $A,B$ be subsets of a locally compact, abelian group $G$ and let $\mu$ 
denote (a version of) the Haar measure on $G$. If there exists an open 
subgroup $H$ such that 
$$
A+B+H=A+B \quad\text{and}\quad \mu(A+B+H) = \mu(A+H) + \mu(B+H)-\mu(H),
$$
then 
$$
\mu(A+B) \geq \mu(A)+\mu(B)-\mu(H).
$$
Otherwise, 
$$
\mu^*(A+B) \geq \mu(A)+\mu(B),
$$
where $\mu^*$ denotes the inner measure induced by $\mu$.
\end{lemma}

\subsection{Proofs} 
\label{proof2sec}

\noindent
{\sc Proof of Theorem~\ref{equalityTh}.}
The inequality follows by positivity of mutual information 
since $h(X+Y) \geq h(X+Y|Y) = h(X)$ and by symmetry $h(X+Y)\geq h(Y)$.

If $X,Y$ are both uniformly distributed on $G$ then their sum 
and difference are also uniform on $G$ and thus $\dR(X,-Y) = \dR(X,Y) = 0$. 
If $\{0\}$ is open, then every singleton is open since the translation map 
is continuous. Hence, every set is open (as a union of singletons), which 
means the topology is discrete. Since the group is also compact it must 
be finite. Therefore, if the random variables are deterministic all 
entropies in the statement
equal $\log \mu(\{0\})$.  

Conversely, assume $\dR(X,-Y) = 0.$ By the triangle inequality for $\dR$,
    $$
    \dR(X,X),\dR(Y,Y) \leq 2\dR(X,-Y)=0.
    $$
   That is, $h(X - X') - h(X) = h(X - X') - h(X - X'|X') = I(X';X-X') = 0$, which implies that $X'$ is independent of $X - X'$.
Thus, the distribution of $X - X'$ conditioned on $\{X' = x\}$ is independent of $x$, i.e., $X - x \stackrel{d}{=} X - x'$ for all $x,x'$ in $\mathrm{supp}(X)$. Equivalently, $X \stackrel{d}{=} X + x$
for all $x \in \mathrm{supp}(X) - \mathrm{supp}(X)$.  In particular,
$\mathrm{supp}(X) =  \mathrm{supp}(X) + x$ for all $x \in \mathrm{supp}(X) - \mathrm{supp}(X)$, whence we see that if $s, x \in \mathrm{supp}(X)-\mathrm{supp}(X)$ then $s+x \in \mathrm{supp}(X)-\mathrm{supp}(X)$. Since $\mathrm{supp}(X)-\mathrm{supp}(X)$ also contains $0$ and inverses, it is a subgroup
and $\mathrm{supp}(X)$ is a coset of this subgroup, while $X$ is uniformly distributed on this coset.
Since $X$ is absolutely continuous with respect to Haar measure, $\mathrm{supp}(X)$ and $\mathrm{supp}(X) - \mathrm{supp}(X)$ have positive Haar measure.  Thus $\mathrm{supp}(X) - \mathrm{supp}(X)$ must be open since it must contain a neighborhood of the identity and is a union of translates of that neighborhood~\cite[Corollary~20.17]{hewitt:bookI}.   

   If $\{0\}$ is not open, then the only open subgroup is $G$ itself and thus $X$ is uniform on $G$.
   Otherwise, $X$ is either uniform on $G$ or on a coset of $\{0\}$, 
i.e., deterministic. 

The same conclusion must hold for $Y$.
\qed       
The following key proposition is analogous to and 
generalizes the discrete entropy result 
of~\cite[Proposition~1.2]{green:25} to the present setting.
For $p \in [0,1],$ let $H_2(p) = -p\log{p}-(1-p)\log{(1-p)}$ denote the 
binary entropy function.

\begin{proposition}\label{revisitedprop}
Let $G$ be a compact and metrizable topological abelian group
equipped with its Borel $\sigma$-algebra, and let $\mu$ denote the Haar 
probability measure on $G$. 
For any $G$-valued random variables $X, Y$ with finite entropies 
there is a non-empty, closed subset $S$ of $G$ of positive Haar 
measure such that, if $U_S$ is uniformly distributed on $S$, then,
for any constant $C\geq 4$,
\begin{equation} 
\dR(U_S, Y) \leq (C+2)\dR(X, Y) + H_2(1 - 2/C), \label{eq:1.5}
\end{equation}
and
\begin{equation}
\log \frac{\mu(S-S)}{\mu(S)} \leq (2C+4)\dR(X, Y) + 2H_2(1 - 2/C).\label{eq:1.6}
\end{equation}
\end{proposition}

\noindent
{\sc Proof.}
We proceed as in the proof of~\cite[Proposition~1.2]{green:25} 
and justify the corresponding steps differently when necessary. 

Suppose $X, Y$ are independent and write $k = \dR(X, Y)$ for convenience. 
Let $C\geq 4$, arbitrary.
We have
\begin{equation*}
h(X-Y) = \frac{1}{2}h(X) + \frac{1}{2}h(Y) + k \leq \frac{1}{2}h(X-Y) 
+ \frac{1}{2}h(Y) + k, 
\end{equation*}
and hence $h(X-Y) - h(Y) \leq 2k$. By the translation invariance 
of entropy, we have
\begin{align}
h(X-Y) - h(Y) &= \int_{G} \int_{G} f_X(x) f_{x-Y}(t)\log\frac{1}{f_{X-Y}(t)} - h(x-Y)\,dx \,dt \nonumber\\
&=\int_{G} f_X(x) D(x-Y \| X-Y)\,dx \label{fDx}\\
& \leq 2k. \label{eq:3.4}
\end{align}
Define
\begin{equation*}
S := \{x \in G : D(x-Y \| X-Y) \leq Ck\}. \label{eq:3.5}
\end{equation*}
Observe that $S$ is closed by the lower-semicontinuity of relative entropy. Let $p := \mathbb{P}(X \in S) = \mathbb{P}(\mathbb{I}_S(X) = 1)$. By Markov's inequality 
and~\eqref{eq:3.4}, it follows that
\begin{equation*}
p = \mathbb{P}(X \in S) \geq 1 - 2/C \geq 1/2. 
\end{equation*}
Note first that, since 
\begin{align*}
    I(X;\mathbb{I}_S(X)) &= H_2(p) \\
    &= h(X) - h(X|\mathbb{I}_S(X)) \\
    &= h(X) - ph(X|X\in S) - (1-p)h(X|X\notin S),
\end{align*}
we have 
\begin{equation} \label{hExpand}
    h(X) = ph(X|X\in S) + (1-p)h(X|X \notin S) + H_2(p).
\end{equation}

Also, since $Y$ is independent of $X$ and $\mathbb{I}_S(X)$, it follows that 
$$h(X-Y|\IS=i) \geq h(Y)\quad\mbox{and}\quad h(X-Y|\IS=i) \geq h(X|\IS=i),$$ 
for $i=0, 1$. Therefore,
\begin{align*}
h(X-Y) &\geq h(X-Y|\IS)\\
& = ph(X-Y|X\in S) + (1-p)h(X-Y|X \notin S) \\
&\geq ph(Y) + \frac{(1-p)}{2}[h(Y)+h(X|X\notin S)]. 
\end{align*}
Combining with~\eqref{hExpand} and using the definition of $k,$ we obtain 
\begin{equation*}
2k \geq p [h(Y) -  h(X|X\in S)] - H_2(p).  
\end{equation*}
Let $V$ be a random variable distributed as $X$ conditioned on $\{X\in S\}$. 
Since
$0\leq D(V\|U_S) = \log\mu(S)-h(X|X\in S),$ we have 
\begin{align}
    \log\mu(S) &\geq  h(Y) - \frac{2}{p}k - \frac{1}{p} H_2(p) \label{volSbound} \\
    &\geq h(Y) -4k - 2H_2\Big(1-\frac{2}{C}\Big), \label{cardSH}
\end{align}
where we used the monotonicity of $H_2(\cdot)$ 
and the fact that $C \geq 4.$

Now we claim that for any random variable $Z$ 
taking values in $S$ and having a density with respect to $\mu$,
\begin{equation} \label{foreveryz}
    h(Y-Z) - h(Y) \leq Ck.
\end{equation}
Combined with~\eqref{volSbound},~\eqref{foreveryz} implies~\eqref{eq:1.5} 
upon taking $Z = U_S$.

To prove the claim, assume $Y,Z$ are independent and note that,
similarly to~\eqref{fDx}, 
we have for any $Z \in S$ that has a density with respect to $\mu$,
\begin{equation} \label{Dwelldef}
h(Y-Z) - h(Y) \leq \int_Sf_Z(z)D(z-Y\|X-Y)\,dz.
\end{equation}
By the definition of $S$, the right-hand side of~\eqref{Dwelldef} 
is at most $Ck$ (and, in particular, finite).

Finally, we establish~\eqref{eq:1.6}. We may define $(Z, Z')$ on $S\times S$, so that $Z-Z'$ is uniform on $S-S$, which is measurable since $S$ is closed. Indeed, since $S-S$ is in fact closed 
and thus a compact subset of $G$, the Kuratowski and Ryll-Nardzewski 
selection theorem~\cite{kuratowski:65} 
(see also~\cite[Theorem~5.2.1]{srivastava:book}) guarantees that there 
exists a measurable selection $F:S-S\to S\times S$. Letting $D$ be 
uniform on $S-S$ we may thus define $(Z,Z') = F(D)$ so that $Z-Z'=D.$ 
Hence, using~\eqref{cardSH} and~\eqref{toprovetr}, 
\begin{align*} 
\log \frac{\mu(S-S)}{\mu(S)} 
&= h(Z-Z') - \log{\mu(S)} \\
&\leq h(Z-Z') -h(Y) +4k+2H_2\Big(1-\frac{2}{C}\Big)\\
&\leq h(Z-Y)+h(Y-Z')-2h(Y)  +4k+2H_2\Big(1-\frac{2}{C}\Big) \\
&\leq 2Ck  +4k+2H_2\Big(1-\frac{2}{C}\Big),
\end{align*}
where the last inequality follows by~\eqref{foreveryz}. Therefore, 
$$\log 
\frac{\mu(S-S)}{\mu(S)} \leq (2C+4)k+2H_2\Big(1-\frac{2}{C}\Big),$$
which is~\eqref{eq:1.6}.
\qed       

Equipped with Proposition~\ref{revisitedprop}, we are ready to prove the stability estimate of Theorem~\ref{stabilityTh}: 

\medskip

\noindent
{\sc Proof of Theorem~\ref{stabilityTh}.}
Assume $X,Y$ are independent and write $k:= h(X-Y)-\frac12h(X)-\frac12h(Y)$ 
as before. Let $C\geq 4$ be a constant depending on $k$ which will be
chosen later. We apply 
Proposition~\ref{revisitedprop} to obtain a closed set $S$ such that 
    \begin{equation*}
    \dR(U_S,Y) \leq (C+2)k + H_{2}\Big(1-\frac{2}{C}\Big),
    \end{equation*}
    and
    \begin{equation}\label{proofSbounds2}
        \log{\frac{\mu(S-S)}{\mu(S)}} \leq (2C+4)k+2H_2\Big(1-\frac{2}{C}\Big).
    \end{equation}
    Under our assumptions, there are no open subgroups except $G$ and 
possibly $\{0\}$. Also $S$ is closed and thus $S-S$ is measurable. Thus, 
if the group is connected, or equivalently $\{0\}$ is not open, 
Kneser's theorem gives  
    \begin{equation} \label{knesertorus2}
        \mu(S-S) \geq \min\bigl\{2\mu(S),1\},
    \end{equation}
    and 
    if $\{0\}$ is open, equivalently if the group is finite.
     \begin{equation} \label{knesertorus0open}
        \mu(S-S) \geq \min\bigl\{2\mu(S)-\mu(\{0\}),1\}.
    \end{equation}

    Consider first the (non-discrete) case where $\{0\}$ is not open. 
    Then $\mu(S) < 1/2$ is impossible if $k$ is small enough, since then we get from~\eqref{proofSbounds2} and~\eqref{knesertorus2}
    \begin{equation*} 
    k \geq \frac{\log2 - 2H_2(1-\frac{2}{C})}{2C+4},
    \end{equation*}
    and optimizing in $C\geq 4$ we obtain the desired contradiction. 
    Therefore, we assume $\mu(S) \geq 1/2. $ Then by~\eqref{proofSbounds2} 
and~\eqref{knesertorus2} again,
    \begin{equation} \label{muSgeqe2}
    \mu(S) \geq e^{-(2C+4)k-2H_2(1-\frac{2}{C})}.
    \end{equation}
    Since $U_{G}-U_S \sim\mathrm{Uni}(G),$
    $$
    \dR(U_S,U_{G}) =  \frac{1}{2}(h(U_{G})-h(U_S)) = -\frac{1}{2}h(U_S).
    $$
    Furthermore, since $Y-U_G \sim \mathrm{Uni}(G),$
    $$
    D(Y\|U_{G}) = -h(Y) = 2\dR(Y,U_{G}).
    $$
    Thus, 
    \begin{align}
        D(Y\|U_{G}) &= 2\dR(Y,U_{G}) \\
        &\leq2\dR(Y,U_{S}) + 2\dR(U_S, U_{G}) \\ \label{steps}
        &\leq (4C+8)k + 4H_2\Big(1-\frac{2}{C}\Big).
    \end{align}
    Choosing $C = \frac{1}{\sqrt{k}}$ and assuming $k \leq \frac{1}{16}$ so that $C\geq 4$, we obtain the claimed result since 
$H_2(1-x) = O(x\log x)$ as $x\to 0.$

    It remains to prove the claim in the case when $\{0\}$ is open. 
Note that in that case the group is discrete (every set is open and since 
the group is compact it is finite), $\mu$ is the normalized counting 
measure and the entropy lies in $[\log \mu(\{0\}),0]$. Then,
if $\mu(S) \geq \frac{1+\mu(\{0\})}{2}$,~\eqref{knesertorus0open} gives 
$\mu(S-S)\geq 1$ and we repeat the steps~\eqref{muSgeqe2}--\eqref{steps} 
with $C=1/\sqrt{k}$ to obtain $D(Y\|U_G) = O(\sqrt{k}\log{k})$. On the 
other hand, if $\mu(S) < \frac{1+\mu(\{0\})}{2}$,~\eqref{proofSbounds2} 
and~\eqref{knesertorus0open} yield 
    $$
     (2C+4)k+2H_2(1-\frac{2}{C}) \geq \log{\Bigl(2-\frac{\mu(\{0\})}{\mu(S)}\Bigr)},
    $$
    or, equivalently, 
     $$
    \mu(S) \leq \frac{\mu(\{0\})}{2-e^{(2C+4)k+2H_2(1-\frac{2}{C})}}.
    $$
    But as in~\eqref{volSbound}, 
    $$
     \mu(S) \geq e^{h(Y) -4k - 2H_2(1-\frac{2}{C})}.
    $$
    Combining the last two inequalities we obtain that,
   $$
    h(Y) \leq \log \mu(\{0\})+\log{\Bigl(\frac{e^{4k + 2H_2(1-\frac{2}{C})}}{2-e^{(2C+4)k+2H_2(1-\frac{2}{C})}}\Bigr)} = \log \mu(\{0\})+ O(\sqrt{k}\log{k}) \text{ as } k \to0,
    $$
     after choosing $C = 1/\sqrt{k}$.

     Clearly all the inequalities above also hold for $X$ in place of $Y,$ 
where $Y$ was arbitrary (independent of $X$). Therefore, we may replace 
$Y$ with $-Y$ and the proof is complete.
\qed

Although we may invert $x\mapsto\sqrt{x}\log{x}$ to get a strict lower bound on the deficit directly from Theorem~\ref{stabilityTh} in the case where $\{0\}$ is not open (equivalently when $G$ is connected), below we record a simple, explicit lower bound that depends on the entropy (equivalently on the relative entropy from the uniform) and is strictly positive. In that sense, it may be seen as an entropy jump statement on compact groups.

\begin{theorem} \label{12jump}
Let $G$ be a connected, compact and metrizable topological abelian group, equipped with its Borel $\sigma$-algebra. Let $X,Y$ be independent random variables on $G$ with finite entropies. 
     There exists a continuous function (independent of $G$, $X$ and $Y$),
     $$
     f:(-\infty,0] \to [0,+\infty), \text{ such that } f(x) = 0 \text{ if and only if } x=0
     $$
     and 
        $$
        \delta_{\mathrm{EPI}} = h(X+Y)-\frac12h(X)-\frac12h(Y) \geq \max{\{f(h(X)),f(h(Y))\}}.
        $$

        In fact, we may take 
         \begin{align}
f(x) =
\begin{cases} 
\frac{-x^5-4x^6}{16+8x^4} & \text{if } -\frac{1}{8}\leq x\leq 0\\ \label{fdef}
f(-\frac{1}{8}) & \text{if } x < -\frac{1}{8}.
\end{cases}
\end{align}
\end{theorem}

As will be seen from the proof, a more general family of bounds
than those corresponding to the specific $f$ in~\eqref{fdef}
are obtained in~\eqref{star}, and~\eqref{fdef} follows by a
specific (simple but suboptimal) value for $C$.
If we are interested in the best rate as $d:=D(Y\|U_G) = -h(Y) \to 0^+$ 
we may choose $C \sim \frac{\log{(\frac{1}{d}})}{d}$ to get a 
bound $\delta_{\rm EPI} \geq \Theta(\frac{d^2}{\log(1/d)})$.

\noindent
{\sc Proof.}
Let $X,Y$ be independent. As before we will show the lower bound for $\dR(X,Y)$ and then replace $Y$ with $-Y$.
    Write $k:= \dR(X,Y)$ and let $C\geq4$ be a constant depending on $k$ 
to be chosen later. Applying Proposition~\ref{revisitedprop}, we obtain a closed set $S$ such that 
    \begin{equation} \label{drBound}
    \dR(U_S,Y) \leq (C+2)k + H_{2}\Big(1-\frac{2}{C}\Big),
    \end{equation}
    and
    \begin{equation}\label{proofSbounds}
        \log{\frac{\mu(S-S)}{\mu(S)}} \leq (2C+4)k+2H_2\Big(1-\frac{2}{C}\Big).
    \end{equation}

    We apply Kneser's theorem to $S-S$. Since $G$ is connected and compact, there are no proper open subgroups by Lemma~\ref{lemmaproper} and hence $\{0\}$ is not open. Also, since $S$ is closed, $S-S$ is measurable and we may replace the inner measure with the Haar measure. Therefore, Kneser's theorem is 
equivalent to
    \begin{equation} \label{knesertorus}
        \mu(S-S) \geq \min\bigl\{2\mu(S),1\}.
    \end{equation}

    If $\mu(S) < 1/2$, then we get from~\eqref{proofSbounds} 
and~\eqref{knesertorus} that
    \begin{equation*} 
    k \geq \frac{\log2 - 2H_2(1-\frac{2}{C})}{2C+4},
    \end{equation*}
    and using $H_2(1-x) \leq x\log(\frac{e}{x}), x \leq \frac{1}{2},$ we have 
    \begin{equation} \label{10minus3}
    k \geq \frac{\log2 - 2H_2(1-\frac{2}{C})}{2C+4} > 3\cdot10^{-3},
     \end{equation}
    by optimizing over $C\geq 4.$

    If $\mu(S) \geq 1/2, $ then by~\eqref{proofSbounds} 
and~\eqref{knesertorus} again we obtain that
    \begin{equation} \label{muSgeqe}
    \mu(S) \geq e^{-(2C+4)k-2H_2(1-\frac{2}{C})}.
    \end{equation}
    
    Denote as before $U_{G}\sim \mathrm{Uni}(G)$
and observe that, since $U_{G}-U_S \sim\mathrm{Uni}(G),$
    $$
    \dR(U_S,U_{G}) =  \frac{1}{2}(h(U_{G})-h(U_S)) = -\frac{1}{2}h(U_S) = -\frac{1}{2}\log\mu(S).
    $$
    Thus, using the triangle inequality,~\eqref{drBound} and~\eqref{muSgeqe}, 
    \begin{align*}
        -\frac{1}{2}h(Y) &= \dR(Y,U_{G}) \\
        &\leq\dR(Y,U_S) + \dR(U_S, U_{G}) \\
        &\leq (2C+4)k + 2H_2(1-\frac{2}{C}), 
    \end{align*}
    hence,
    \begin{equation} \label{star}
      k \geq \frac{-\frac{1}{2}h(Y)-2H_2(1-\frac{2}{C})}{2C+4}  \geq     \frac{-\frac{1}{2}h(Y)-\frac{4}{C}\log{(\frac{eC}{2})}}{2C+4} .
    \end{equation}

    Combining~\eqref{10minus3} and~\eqref{star}, we have established 
    \begin{equation*} 
    \dR(X,Y) \geq \min\Big\{3\cdot10^{-3}, \sup_{C\geq 4}\frac{-\frac{1}{2}h(Y)-\frac{4}{C}\log{(\frac{eC}{2})}}{2C+4}\Big\}.
    \end{equation*}

    Optimizing~\eqref{star} over $C\geq 4$ and recalling~\eqref{10minus3}, 
it is clear that there exists a continuous function $g:\mathbb{R}_-\mapsto\mathbb{R}_+$ with $g(x)>0$ for $x< 0$, such that $\dR(X,Y) \geq g(h(Y))$. 
Next we plug in a convenient value for $C$ to get the claimed explicit 
function $f$ 
in~\eqref{fdef}. 
    
    If $h(Y) < -\frac{1}{8},$ substituting this into~\eqref{star} and 
optimizing over $C\geq 4, $ we get 
    \begin{equation} \label{10minus5}
        k \geq 10^{-5}. 
    \end{equation}
    On the other hand, if $h(Y) \geq -\frac{1}{8}$, we choose $C = \frac{4}{h(Y)^4} > 8$ and use $\log(2x) \leq \sqrt{x}, x > 8,$ to obtain 
    \begin{equation} \label{finalkboundnum}
    k \geq \frac{-\frac{1}{2}h(Y) - \frac{4}{\sqrt{C}}}{2C+4}=\frac{-\frac{1}{2}h(Y) - 2h(Y)^2}{\frac{8}{h(Y)^4}+4} = \frac{-h(Y)^5 - 4h(Y)^6}{16+8h(Y)^4}.
    \end{equation}

    Comparing~\eqref{10minus3},~\eqref{10minus5} and~\eqref{finalkboundnum}, we see that $f(-1/8) < 10^{-5}$ and thus,
    $$
    \dR(X,Y) = k\geq f(h(Y)),
    $$
    with $f$ as defined in~\eqref{fdef}. But since $\dR(X,Y)$ is symmetric, the last bound also hold for $h(X)$ in place of $h(Y)$, so we may choose the maximum. This completes the proof. 
\qed

\section{Entropy jumps in connected compact groups} 
\label{connectedSec}
 
We start by stating and proving the simple consequence 
of Theorem~\ref{maxentropyepi} for rates of convergence in compact groups. 

An analogous result to Corollary~\ref{rates} was established 
in \cite[Theorem~1]{johnson:00b}. That result is more general in 
that it holds for any compact group. On the other hand, in the case 
of compact metrizable abelian groups and random variables that 
have densities, it requires that the densities are bounded below 
in order for the result to be meaningful, i.e., for the rate 
to be strictly less than $1$.  

\begin{corollary}[Entropic rates of convergence on compact groups] 
\label{rates}
Let $X_1,\ldots,X_n$ be independent random variables in a connected, 
compact and metrizable abelian group $G$. Let $S_n = \sum_{i=1}^nX_i$
and write $j^*$ for the index that achieves
$$h(X_{j^*})=\max_{1\leq i\leq n}h(X_i),$$
with any ties broken arbitrarily. Then
$$
D(S_n\|U_G)
\leq
\Big[
\prod_{1\leq i\leq n, i\neq j^*}r(h(X_i))\Big]
D(X_{j^*}\|U_G)
\leq
r\Bigl(\min_{1\leq i\leq n} h(X_i)\Bigr)^{n-1}
\min_{1\leq i\leq n}D(X_i\|U_G),
$$
where
$$
r(x)
=
1-\min\left\{
\frac14,
\frac{\bigl(1-\operatorname{sinc}(e^{6x-2})^2\bigr)^2}
{8+16\frac{\operatorname{sinc}(e^{6x-2})^2}{e^{6x-2}}}
\right\},
\qquad x\in(-\infty,0],
$$
and as usual $\mathrm{sinc}(x) = \frac{\sin(\pi x)}{\pi x}, x \in \mathbb{R}.$
\end{corollary}

\noindent
{\sc Proof.}
Write $S_{n\backslash \{i\}} = \sum_{j\neq i, 1\leq j \leq n}X_i$. 
Observe that the bound in Theorem~\ref{maxentropyepi} is non-decreasing in $s$. Applying it to $Y = S_{n\backslash \{i\}}$ and $X = X_i$, where 
$i = \mathrm{argmin}_{1\leq j\leq n} h(X_j)$ gives 
$$
D(S_n\|U_G) \leq r(h(X_i))D(S_{n \backslash i}\|U_G) .
$$
Repeating the bound while leaving out the minimum entropy variable each time and using the non-increasing nature of $r(\cdot)$ gives the claim. 
\qed       

The proof of Theorem~\ref{maxentropyepi} is similar to that of 
Theorem~\ref{stabilityTh}, up to the point where we bound the deficit 
below by a deficit involving the uniform on the set $S$. When the group 
is connected, we finish the argument by the strong data processing 
estimate obtained in Section~\ref{SDPIsection} using harmonic-analytic 
arguments. 

\medskip

\noindent
{\sc Proof of Theorem~\ref{maxentropyepi}.}
Assume $X,Y$ are independent and $h(Y)\geq h(X).$
Similarly to the previous section, let $\kk = h(X-Y) - h(Y)$.
By the translation invariance of entropy, we have
\begin{align}
\kk &=h(X-Y) - h(Y)\nonumber \\
&= \int_{G} \int_{G} f_X(x) f_{x-Y}(t)\log\frac{1}{f_{X-Y}(t)} - h(x-Y)\, dx\,  dt \nonumber\\
&=\int_{G} f_X(x) D(x-Y \| X-Y)\, dx. \label{fDx1} 
\end{align}
Define
\begin{equation*}
S := \{x \in G : D(x-Y \| X-Y) \leq 2\kk\}.
\end{equation*}
Note that $S$ is closed by the lower-semicontinuity of relative entropy. 
Let $p := \mathbb{P}(X \in S) = \mathbb{P}(\mathbb{I}_S(X) = 1)$. 
By Markov's 
inequality and~\eqref{fDx1}, it follows that
\begin{equation*}
p = \mathbb{P}(X \in S)\geq 1/2.
\end{equation*}
Note first that since 
\begin{align*}
    I(X;\mathbb{I}_S(X)) &= H_2(p) \\
    &= h(X) - h(X|\mathbb{I}_S(X)) \\
    &= h(X) - ph(X|X\in S) - (1-p)h(X|X\notin S),
\end{align*}
we have 
\begin{equation} \label{hExpand1}
    h(X) = ph(X|X\in S) + (1-p)h(X|X \notin S) + H_2(p).
\end{equation}

Also, since $Y$ is independent of $X$ and $\mathbb{I}_S(X)$, it follows 
that, for $i=0,1$,
$$
h(X-Y|\IS=i) \geq 
h(Y)\quad\mbox{and}\quad 
h(X-Y|\IS=i) \geq 
h(X|\IS=i).$$ 
Therefore,
\begin{align*}
h(X-Y) &\geq h(X-Y|\IS)\\
& = ph(X-Y|X\in S) + (1-p)h(X-Y|X \notin S)  \\
&\geq ph(Y) + \frac{(1-p)}{2}[h(Y)+h(X|X\notin S)]. 
\end{align*}
Combining with~\eqref{hExpand1} and using the definition of $\kk,$ we obtain 
\begin{equation*}
\kk \geq -\frac{1-p}{2}h(Y) -  
\frac{p}{2}h(X|X\in S) +\frac{1}{2}h(X) -\frac{1}{2}H_2(p).
\end{equation*}
Let $V$ be a random variable distributed as $X$ conditioned on $\{X\in S\}$. 
Since $0\leq D(V\|U_S) = \log\mu(S)-h(X|X\in S),$ we have 
$$
\log\mu(S) \geq  -\frac{1-p}{p}h(Y) +\frac{1}{p}h(X) - \frac{1}{p} H_2(p) -\frac{2}{p}\kk.
$$
Since $h(X),h(Y)\leq 0,$ $\kk = h(X-Y)-h(Y) \leq -h(Y)$ and $p\geq \frac{1}{2},$ this implies
\begin{align} 
    \log\mu(S) &\geq  4h(Y) +2h(X) - 2 \geq 6h(X)-2. \label{volSbound1} 
\end{align}
so that, in particular, $\mu(S)>0, $ as long as $h(X) >-\infty.$

By the expression~\eqref{fDx1} and the definition of $S$,
\begin{equation} \label{relativetostrong}
2\kk \geq h(Y-U_S) - h(Y).
\end{equation}
If $h(Y-U_S) > \frac{1}{2}h(Y)$ we are done, since then
$$
2\kk \geq h(Y-U_S)-h(Y) \geq -\frac{1}{2}h(Y).
$$
So we may assume that 
\begin{equation} \label{bootstrap}
    h(Y-U_S) \leq \frac{1}{2}h(Y).
\end{equation}
Let $U_S'$ be an independent copy of $U_S$. By an application of the sum-submodularity inequality, Proposition~\ref{submodularity}, we have 
\begin{align}
    h(Y-U_S) - h(Y) &\geq h(Y-U_S'-U_S) - h(Y-U_S') \nonumber\\ 
    &= D(Y-U_S'\|U_{G}) - D(Y-U_S'+U_S\|U_{G}).
\label{repeatedUS}
\end{align}
Since 
$f_{Y-U_S'}(y) = \frac{1}{\mu(S)}\int_Sf_Y(y+u)\,du \leq \frac{1}{\mu(S)},$ 
we may apply the strong data processing inequality for relative entropy
(Proposition~\ref{DSDPIlemma} in Section~\ref{SDPIsection} below)
to $-(Y-U_S')$, 
to obtain that
\begin{align*}
 D(Y-U_S'\|U_{G}) - D(Y-U_S'-U_S\|U_{G}) &\geq \frac{(1-\eta_S)^2}{2+4\eta_S\|f_{Y-U_S'}\|_{\infty}}D(Y-U_S\|U_{G}) \\
 &\geq \frac{(1-\eta_S)^2}{2+\frac{4\eta_S}{\mu(S)}}D(Y-U_S\|U_{G}),
\end{align*}
where 
$$\eta_S = \left(\frac{\sin(\pi \mu(S))}{\pi \mu(S)}\right)^2.$$
But now the assumption~\eqref{bootstrap} says 
that $D(Y-U_S\|U_{G}) \geq \frac{1}{2}D(Y\|U_{G})$, and 
combining with~\eqref{repeatedUS}, gives
\begin{equation}\label{etasfinalbound}
h(Y-U_S) - h(Y) \geq \frac{(1-\eta_S)^2}{4+\frac{8\eta_S}{\mu(S)}}D(Y\|U_{G}),
\end{equation}
where $\mu(S)$ satisfies the bound~\eqref{volSbound1}. The result follows 
by~\eqref{relativetostrong} after noting that the bound 
in~\eqref{etasfinalbound} is non-decreasing in $\mu(S)$. 
\qed       

\subsection{Strong data processing via harmonic analysis} 
\label{SDPIsection}

Let $P,Q$ be probability measures. The $\chi^2$ divergence between $P$ and $Q$ is 
$$
\chi^2(P\|Q) := \int\Big(\frac{dP}{dQ}-1\Big)^2\,dQ  \quad \text{if } P \ll Q
$$
and $\chi^2(P\|Q) = \infty$ otherwise. 

When $P$ and $Q$ have densities, say $f$ and $g$, with respect to a common measure, e.g., the Haar measure, we also write 
$\chi^2(f\|g) = \chi^2(P,Q).$

We use $f\ast g$ to denote the convolution of $f$ and $g$ on $G$. We denote the Pontryagin dual of $G$, i.e. the space of all continuous homomorphisms from $G$ to the one-dimensional torus $\mathbb T=\RL/(2\pi\IN)$, 
by $\hat{G}.$ We also identify ${\mathbb T}$ with the unit-circumference
circle in the complex plane $\Co$ whenever convenient.
We refer the reader to the 
textbook~\cite{hewitt:bookI,hewitt:bookII} for harmonic analysis 
on such groups. 

First we need an upper bound on the absolute value of nontrivial Fourier 
coefficients. In the proof we rely on the observation that nontrivial 
characters in connected compact abelian groups must be surjective, which 
was also used in~\cite{tao:18} -- see the discussion following 
equation~(4.21) there. 

\begin{lemma} 
\label{nonzerocharlemma}
Let $G$ be a connected compact abelian group, let $\mu$ be its Haar 
probability measure, and let $\chi \in \widehat G$ be a nontrivial 
continuous character. For a measurable set $S \subset G$ with 
$\mu(S)\in(0,1],$ write $g_S:=\frac{ \mathbb{I}_S}{\mu(S)}.
$
Then:
$$ 
\bigl|\widehat{g_S}(\chi)\bigr|
=
\left|
\frac{1}{\mu(S)}\int_G   \mathbb{I}_S(x)\,\overline{\chi(x)}\,d\mu(x)
\right|
\leq
\frac{\sin(\pi\mu(S))}{\pi\mu(S)}.
$$ 
\end{lemma}

\noindent
{\sc Proof.}
Since $G$ is connected and $\chi$ is continuous, the image $\chi(G)$ is a compact and connected subset of $\mathbb T$. Since $\chi$ is a homomorphism, $\chi(G)$ is also a subgroup, and because $\chi\neq 1$ it must be nontrivial. 
The only such subgroup is $\mathbb T$ itself and hence $\chi(G)=\mathbb T$.
Consider the pushforward measure $\nu$ defined by $\nu(A) = \mu(\chi^{-1}(A))$ for every Borel measurable $A \subset \mathbb T$. By translation invariance of $\mu$, $\nu$ is rotation invariant on $\mathbb T$ and thus a Haar measure. Since $\chi(G)=\mathbb T$, $\nu(\mathbb T) = 1$ and therefore it is the normalized Haar measure on $\mathbb T$. 
Since $\chi(G) = \mathbb T$, we may define $\phi_S:\mathbb T\to [0,1]$ by $\phi_S = \frac{d\nu_S}{d\nu},$ where $\nu_S(A) = \mu(\chi^{-1}(A) \cap S).$ One can then check that $\phi_S \circ\chi$ is a version of 
the conditional expectation 
$\mathbb{E}(\mathbb I_S|\chi)=
\mathbb{E}(\mathbb I_S|\sigma(\chi))$. 
Hence, by a change of variables,
$$
\int_G \mathbb I_S(x)\,\overline{\chi(x)}\,d\mu(x)
=
\int_G \mathbb E(\mathbb I_S\mid \chi)(x)\,\overline{\chi(x)}\,d\mu(x)
=
\int_{G}\phi_S\circ\chi(x)\,\overline{ \chi(x)}\,d\mu(x)
=
\int_{\mathbb T}\phi_S(z)\,\overline z\,d\nu(z).
$$

Therefore, it suffices to show that
$$
\left|\int_{\mathbb T}\phi_S(z)\,\overline z\,d\nu(z)\right|
\leq
\frac{\sin(\pi\mu(S))}{\pi}.
$$
We note that 
$$
\int_{\mathbb T}\phi_S\,d\nu
=
\int_G \mathbb E(\mathbb I_S| \chi)\,d\mu
=
\int_G \mathbb I_S\,d\mu
=
\mu(S).
$$
Choosing $\eta\in\mathbb T$ such that
$$
\Big|\int_{\mathbb T}\phi_S(z)\,\overline z\,d\nu(z)\Big| = \eta\int_{\mathbb T}\phi_S(z)\,\overline z\,d\nu(z),
$$
we obtain
$$
\Big|\int_{\mathbb T}\phi_S(z)\,\overline z\,d\nu(z)\Big|
=
\Re\Bigl(\eta \int_{\mathbb T}\phi_S(z)\,\overline z\,d\nu(z)\Bigr)
=
\int_{\mathbb T}\phi_S(z)\,\Re(\eta \overline z)\,d\nu(z),
$$
where we write $\RL w$ for the real part of any complex $w\in\Co$.
And since $\phi_S \leq 1,$ we have
$$
\phi_S(z)=\int_0^1 \mathbb I_{\{\phi_S>t\}}(z)\,dt,
$$
so that, by Fubini's theorem, we get
$$
\Big|\int_{\mathbb T}\phi_S(z)\,\overline z\,d\nu(z)\Big|
=
\int_0^1 \left(\int_{A_t}\Re(\eta\overline z)\,d\nu(z)\right)\,dt,
$$
where $A_t:=\{z:\phi_S(z)>t\}.$
Using Fubini's theorem again,
$$
\int_0^1 \nu(A_t)\,dt
=
\int_{\mathbb T}\phi_S\,d\nu
=
\mu(S).
$$

Now we claim, that for every measurable $A\subset\mathbb T$,
$$
\int_A \Re(\eta\overline z)\,d\nu(z)\leq \frac{\sin(\pi \nu(A))}{\pi}.
$$
Indeed, by rotation invariance of $\nu$, it suffices to treat 
the case $\eta=1$, so that the integrand is $\Re z$. 
Writing $z=e^{i\theta}$, $\theta\in[-\pi,\pi)$, and 
$d\nu(z)=\frac{d\theta}{2\pi}$, we observe that for
any $a\in[0,1]$ the supremum
$$
\sup_{|A|=2\pi a}\frac{1}{2\pi}\int_A \cos\theta\,d\theta
$$
 is achieved by $A=[-\pi a,\pi a].$
For this arc,
$$
\frac{1}{2\pi}\int_A \cos\theta\,d\theta
=
\frac{1}{2\pi}\int_{-\pi a}^{\pi a}\cos\theta\,d\theta
=
\frac{\sin(\pi a)}{\pi}.
$$
This proves the claim.

Applying the claim to $A=A_t$ gives
$$
\Big|\int_{\mathbb T}\phi_S(z)\,\overline z\,d\nu(z)\Big|
\leq
\int_0^1 \frac{\sin(\pi \nu(A_t))}{\pi}\,dt.
$$
Since $x\mapsto \sin(\pi x)$ is concave on $[0,1]$, 
$$
\Big|\int_{\mathbb T}\phi_S(z)\,\overline z\,d\nu(z)\Big|
\leq
\frac{1}{\pi}\sin\!\left(\pi\int_0^1\nu(A_t)\,dt\right)
=
\frac{\sin(\pi\mu(S))}{\pi}.
$$
Therefore,
$$
\bigl|\widehat{g_S}(\chi)\bigr|
=
\frac{\Big|\int_{\mathbb T}\phi_S(z)\,\overline z\,d\nu(z)\Big|}{\mu(S)}
\leq
\frac{\sin(\pi\mu(S))}{\pi\mu(S)},
$$
which is the desired inequality.
\qed       

The previous estimate naturally leads to the following $L^2$ contraction estimate:
\begin{lemma} \label{fourieruniformLemma}
Let $G$ be a connected compact abelian group, let $\mu$ be its 
Haar probability measure and $Y$ be a random variable on $G$ with density $f_Y \in L^2(G)$. Let $S$ be a set of positive Haar measure and 
write $g_S = \frac{1}{\mu(S)}\mathbb{I}_S$ for the density of the 
random variable $U_S$ having the uniform distribution on $S$ and 
being independent of $Y$. Then,
\begin{equation} \label{L2eta}
    \int_{G}(f_Y\ast g_S(x) - 1)^2\, dx \leq \eta_S\int_{G}(f_Y(x)-1)^2\, dx,
\end{equation}
where,
$$\eta_S = \left(\frac{\sin(\pi \mu(S))}{\pi \mu(S)}\right)^2.$$
Equivalently, 
\begin{equation} \label{chi2claim}
\chi^2(Y+U_S\,\|\,U_{G}) = \chi^2(Y+U_S\,\|\,U_{G} +  U_S)
\le
\eta_S\,\chi^2(Y\,\|\,U_{G}),
\end{equation}
where
$U_{G}$ is uniformly distributed on $G$ and independent of $U_S$.

\end{lemma}

\noindent
{\sc Proof.}
Setting
$
h:=f_Y-1,$
we have $\int_{G} h(x)\,dx=0$, and
$$
\chi^2(Y\,\|\,U_{G})
=\int_{G} (f_Y-1)^2(x)\,dx
=\|h\|_{L^2(G)}^2.
$$
Since the density of $Y+U_S$ is $f_Y*g_S$ and $1*g_S=1$, we have
$$
(f_Y*g_S)-1=(f_Y-1)*g_S=h*g_S.
$$
Therefore,
$$
\chi^2(Y+U_S\,\|\,U_{G})
=\|h*g_S\|_{L^2(G)}^2,
$$
and so~\eqref{L2eta} and~\eqref{chi2claim} are equivalent 
since $U_{G}+U_S$ has the same distribution
as $U_{G}$ and has density $\mathbb{I}_{G}$.

By Parseval's identity~\cite[Corollary~31.19]{hewitt:bookI},
$$
\|h*g_S\|_{L^2( G)}^2
=
\int_{\hat{G}}|\widehat{h\ast g_S}(\chi)|^2\,d\theta(\chi)
=
\int_{\hat{G}} |\widehat{h}(\chi)|^2\,|\widehat{g_S}(\chi)|^2\,d\theta(\chi),
$$
where $\theta$ is Haar measure on $\hat{G}$,
$$
\widehat{g_S}(\chi)=\int_{G} g_S(x)\overline{\chi(x)}\,dx
=\frac1{\mu(S)}\int_S \overline{\chi(x)}\,dx
$$
are the Fourier coefficients of $g_S$, and similarly 
$\widehat{h}$ are the Fourier coefficients of $h$.
Since $\widehat h(1)=0$, 
\begin{equation*} 
\|h*g_S\|_{L^2(G)}^2
=
\int_{\hat{G}} |\widehat{h}(\chi)|^2\,|\widehat{g_S}(\chi)|^2\,d\theta(\chi)
\le
\Big(\sup_{\chi \neq 1} |\widehat{g_S}(\chi)|^2\Big)
\int_{\hat{G}} |\widehat{h}(\chi)|^2\,d\theta(\chi).
\end{equation*}
Applying Parseval's identity again,
\begin{equation*} 
\int_{\hat{G}} |\widehat{h}(\chi)|^2\,d\theta(\chi)=\|h\|_{L^2(G)}^2.
\end{equation*}
The result follows by Lemma~\ref{nonzerocharlemma}.
\qed       

The following lemma is essentially~\cite[Eq.~(78)]{polyanskiy-wu:17}, 
but taking into account the error term. 
See also \cite[Theorem~1]{choi:94}. It holds in any locally 
compact abelian group. 

\begin{lemma} 
\label{Dchilemma}
    Let $G$ be any locally compact abelian group. Let $P,Q$ be probability measures on $G$ that are absolutely continuous with respect to Haar measure, 
and such that $D(P\|Q)<\infty.$ For any $t\geq 0$, 
write $Q^t$ for the probability measure
    $$
    Q^t:= \frac{tQ+P}{1+t}.
    $$
    Then for any $T>0$:
    $$
    \int_0^T \chi^2(P\|Q^t)\,\frac{dt}{t(1+t)} = D(P\|Q) - \int\log{\Biggl(\frac{\frac{dP}{dQ} + T}{1+T}\Biggr)}\, dP. 
    $$
\end{lemma}

\noindent
{\sc Proof.}
Since $D(P\|Q)<\infty$, $P$ is absolutely continuous with respect to $Q$
with density $f$, say. By definition, $Q^t$ is also absolutely continuous
with respect to $Q$ for any $t>0$, with a positive density 
$f_t=\frac{f+t}{1+t}$.
Let $T>0$ arbitrary. For any $t>0$, we have
\begin{align*}  
    \int \frac{f-1}{(f+t)(1+t)}\,dP 
    & =  \int \frac{f-1}{f_t(1+t)^2}\,dP \\
    &=   \frac{1}{t(1+t)}\int \frac{(f-1)\frac{t}{1+t}}{f_t}\,dP  \\
    &=   \frac{1}{t(1+t)}\int \frac{f-f_t}{f_t}\,dP \\
    &=   \frac{1}{t(1+t)}\int \Big(\frac{f}{f_t}-1\Big)
		\frac{f}{f_t}\,dQ^t \\
    &=   \frac{1}{t(1+t)}\int \Big(\frac{f}{f_t}-1\Big)^2\,dQ^t\\
    &=  \frac{1}{t(1+t)}\chi^2(P\|Q^t).
\end{align*}
Integrating over $t$ and using Fubini's theorem,
\be
\int_0^T 
\chi^2(P\|Q^t)
\frac{dt}{t(1+t)}
=
    \int \int_0^T\frac{f-1}{(f+t)(1+t)}\,dt\,dP.
\label{eq:identity}
\ee
Now we observe that, for any $x>0, $ 
\be
    \int_0^T\frac{x-1}{(x+t)(1+t)}\,dt = \log{x} 
	- \log{\Bigl(\frac{x+T}{1+T}\Bigr)}.
\label{eq:integral}
\ee
Taking $x=f$ whenever $f$ is positive, i.e., on the support
of $P$, and substituting~\eqref{eq:integral} into~\eqref{eq:identity},
gives
$$
\int_0^T 
\chi^2(P\|Q^t)
\frac{dt}{t(1+t)}
=
D(P\|Q)-
\int
\log\Big(\frac{f+T}{1+T}\Big)
\,dP,
$$
as required.
\qed       

Next we use Lemma~\ref{fourieruniformLemma} to bound the 
interpolating chi-squared divergences in the integrand. 
\begin{lemma} \label{gtsdpilemma}
    Let $G$ be a connected compact abelian group, let $\mu$ be its 
Haar probability measure and $Y$ be a random variable on $G$ with 
density $f_Y$. Let $S$ be a set of positive Haar measure 
and $g_S = \frac{1}{\mu(S)}\mathbb{I}_S$ be the density of a uniform 
random variable $U_S$ on $S$. For $t>0$, also define
    \begin{equation*} 
    g^t(y) = \frac{t+f_Y(y)}{1+t}, \quad y \in G.
    \end{equation*}
    Then, for all $t > 0 $
\begin{equation*} 
    \chi^2(f_Y \ast g_S \| g^t  \ast g_S) \leq \eta_S\frac{t+\|f_Y\|_{\infty}}{t}\chi^2(f_Y\|g^t),
\end{equation*}
where
$$\eta_S = \left(\frac{\sin(\pi \mu(S))}{\pi \mu(S)}\right)^2.$$

\end{lemma}

\noindent
{\sc Proof.}
    Denote $M := \|f_Y\|_{\infty}$.
We have 
\begin{equation} \label{chi2M}
    \chi^2(f_Y\|g^t) = \int_{G}\frac{(f_Y(y)-g^t(y))^2}{g^t(y)} \, dy = \frac{t^2}{t+1}\int_{G}\frac{(f_Y(y)-1)^2}{t+f_Y(y)} \, dy \geq \frac{t^2}{(t+1)(t+M)}\int_{G}{(f_Y(y)-1)^2} \, dy .
\end{equation}
Furthermore, 
\begin{align*}
     \chi^2(f_Y \ast g_S\|g^t \ast g_S) 
&= \int_{G}\frac{\big((f_Y-g^t)\ast g_S\big)^2(y)}{g^t\ast g_S(y)}\,dy \\
     &= \frac{t^2}{t+1}\int_{G}\frac{\big(f_Y\ast g_S-1\big)^2(y)}
{t+f_Y\ast g_S(y)} \, dy \\
     &\leq \frac{t}{t+1}\int_{G}{\big(f_Y\ast g_S-1\big)^2(y)}\, dy.
\end{align*}
Using Lemma~\ref{fourieruniformLemma}, we obtain 
\begin{equation*}
     \chi^2(f_Y \ast g_S\|g^t \ast g_S) \leq \frac{t}{t+1}\eta_S \int_{G}{\big(f_Y-1\big)^2(y)}\, dy,
\end{equation*}
and combining this bound with~\eqref{chi2M} gives,
$$
 \chi^2(f_Y \ast g_S\|g^t \ast g_S) \leq \eta_S\frac{t+M}{t} \chi^2(f_Y \|g^t ),
$$
which is the claimed inequality. 
\qed       

We are finally in position to prove the required 
strong data processing inequality for relative entropy: 

\begin{proposition} \label{DSDPIlemma}
Let $G$ be a connected compact abelian group and let $\mu$ be its 
Haar probability measure. Let $Y$ be a random variable on $G$ with bounded density $f_Y$. Let $S$ be a set of positive Haar measure and 
write $g_S = \frac{1}{\mu(S)}\mathbb{I}_S$ for the density of 
the random variable $U_S$ with uniform distribution on $S$,
independent of $Y$. Then 
$$
D(Y\|U_{G}) - D(Y+U_S\|U_{G}) \geq \frac{(1-\eta_S)^2}{2+4\eta_S\|f_Y\|_{\infty}}D(Y\|U_{G}),
$$
where 
$$\eta_S = \left(\frac{\sin(\pi \mu(S))}{\pi \mu(S)}\right)^2.$$
\end{proposition}

\noindent
{\sc Proof.}
Let $T>0$ be a large constant to be chosen later. As in 
Lemma~\ref{Dchilemma}, for each $t>0$ we 
write $g^t(y) = \frac{t+f_Y(y)}{1+t}, y \in G.$ Note that $Y+U_S$ has 
density $f_Y\ast g_S$ and 
    $$
    g^t\ast g_S = \frac{t+f_Y\ast g_S}{1+t}.
    $$
    Thus, using the expression obtained from Lemma~\ref{Dchilemma} and noting that $U_{G}+U_{S}$ has density $\mathbb{I}_{G} \equiv 1$ on $G$, we have 
\begin{align*}
    D(Y\|U_{G}) - D(Y+U_S\|U_{G}) &=  D(Y\|U_{G}) - D(Y+U_S\|U_{G}+U_S)\\
   &= \int_0^\infty \chi^2(f_Y\|g^t) - \chi^2(f_Y\ast g_S\|g^t\ast g_S)
\,\frac{dt}{t(1+t)}.
\end{align*}
By the data processing property of $\chi^2$ divergence viewed as
an $f$-divergence~\cite{csiszar:72,liese-vajda:book}, the integrand 
is nonnegative and hence
\begin{align*}
    D(Y\|U_{G}) - D(Y+U_S\|U_{G}) \geq \int_T^\infty \chi^2(f_Y\|g^t) - \chi^2(f_Y\ast g_S\|g^t\ast g_S)\,\frac{dt}{t(1+t)}.
\end{align*}
Applying 
Lemma~\ref{gtsdpilemma}, we obtain 
\begin{align*}
    D(Y\|U_{G}) - D(Y+U_S\|U_{G}) &\geq \Bigl(1-\eta_S\frac{T+\|f_Y\|_{\infty}}{T}\Bigr)\int_{T}^{\infty}\chi^2(f_Y\|g^t)\,\frac{dt}{t(1+t)} \\
    &= \Bigl(1-\eta_S\frac{T+\|f_Y\|_{\infty}}{T}\Bigr)\Bigl(D(Y\|U_{G})-\int_{0}^{T}\chi^2(f_Y\|g^t)\,\frac{dt}{t(1+t)}\Bigr).
\end{align*}
Then by Lemma~\ref{Dchilemma} we have, 
$$D(Y\|U_{G})-\int_{0}^{T}\chi^2(f_Y\|g^t)\,\frac{dt}{t(1+t)} 
= \int_{G}f_Y(y)\log{\Bigl(\frac{f_Y(y)+T}{1+T}\Bigr)} \, dy,$$
and therefore,
\begin{align*}
D(Y\|U_{G}) - D(Y+U_S\|U_{G}) &\geq \Bigl(1-\eta_S\frac{T+\|f_Y\|_{\infty}}{T}\Bigr)\int_{G}f_Y(y)\log{\Bigl(\frac{f_Y(y)+T}{1+T}\Bigr)} \, dy \\
&\geq \frac{1}{1+T}\Bigl(1-\eta_S\frac{T+\|f_Y\|_{\infty}}{T}\Bigr)\int_{G}{f_Y(y)\log{f_Y(y)} \, dy},
\end{align*}
where we used the concavity of the logarithm in the last inequality,
and we assumed $T$ is large enough so that the coefficient is positive. 
But now, 
$$\int_{G}f_Y(y)\log{{f_Y(y)} \, dy} = -h(Y) = D(Y\|U_{G}),$$ 
and choosing $T = 2\eta_S\frac{\|f_Y\|_{\infty}}{1-\eta_S}$ (so 
that the above positivity condition is also satisfied) gives,
\begin{align*}
    D(Y\|U_{G}) - D(Y+U_S\|U_{G}) &\geq \frac{1-\eta_S}{2(1+2\eta_S\frac{\|f_Y\|_{\infty}}{1-\eta_S})}D(Y\|U_{G}) \\
    &= \frac{(1-\eta_S)^2}{2(1-\eta_S+2\eta_S\|f_Y\|_{\infty})}D(Y\|U_{G})  \\
    &\geq \frac{(1-\eta_S)^2}{2+4\eta_S\|f_Y\|_{\infty}}D(Y\|U_{G}),
\end{align*}
where in the last inequality we used that $\eta_S\geq0$.
This is exactly the claimed bound.
\qed       

\section*{Statement on the use of AI}
ChatGPT 5.4 was used in Section \ref{SDPIsection} as an auxiliary tool. In particular, it suggested bounding \eqref{repeatedUS} via the SDPI
for $\chi^2$, and produced an argument for proving Lemma \ref{nonzerocharlemma}. The corresponding proofs appearing in the paper were written by the authors. 

\bibliographystyle{plain}
\bibliography{ik}

\end{document}

%% file: paper_defns.tex
\newcommand{\RL}{{\mathbb R}}

\newcommand{\IN}{{\mathbb Z}}

\newcommand{\BBP}{{\mathbb P}}

\def\ba{\begin{align}}
\def\ea{\end{align}}
\def\ban{\begin{align*}}
\def\ean{\end{align*}}

\def\be{\begin{eqnarray}}
\def\ee{\end{eqnarray}}
\def\ben{\begin{eqnarray*}}
\def\een{\end{eqnarray*}}

\def\bqq{\begin{equation}}
\def\eqq{\end{equation}}
\def\bqqn{\begin{equation*}}
\def\eqqn{\end{equation*}}

\def\sq{$\Box$}

\def\qed{\ifmmode\sq\else{\unskip\nobreak\hfil
\penalty50\hskip1em\null\nobreak\hfil\sq
\parfillskip=0pt\finalhyphendemerits=0\endgraf}\fi\par\medbreak}

\newsavebox{\junk}
\savebox{\junk}[1.6mm]{\hbox{$|\!|\!|$}}

\def\limsup{\mathop{\rm lim\ sup}}
\def\liminf{\mathop{\rm lim\ inf}}

\newcommand{\field}[1]{\mathbb{#1}}

\def\Re{\field{R}}

\def\Co{\field{C}}

\def\til={{\widetilde =}}

\def\clG{{\cal G}}

 \def\eq#1/{(\ref{#1})}

\newtheorem{theorem}{Theorem}[section]
\newtheorem{corollary}[theorem]{Corollary}
\newtheorem{proposition}[theorem]{Proposition}
\newtheorem{lemma}[theorem]{Lemma}
\newtheorem{definition}[theorem]{Definition}

\def\eq#1/{(\ref{e:#1})}

\def\bdes{\begin{description}}
\def\edes{\end{description}}

\def\notes#1{}

\definecolor{mag}{rgb}{0.7,0,0.3}
\definecolor{dgreen}{rgb}{0.1,0.5,0.1}
\definecolor{dred}{rgb}{.8,0,0}
\definecolor{gray}{rgb}{.8,.8,.8}
\definecolor{brown}{rgb}{0.6451,0.3706,0.1745}

%% file: ik.bib
@article{KM:14,
  title={Sumset and inverse sumset inequalities for differential entropy
        and mutual information},
  author={Kontoyiannis, I. and Madiman, M.},
  journal={IEEE Trans. Inform. Theory},
  volume={60},
  number={8},
  pages={4503-4514},
  month={August},
  year={2014},
}

@article{KM:16,
  title={Entropy bounds on abelian groups and the {Ruzsa} divergence},
  author={Madiman, M. and Kontoyiannis, I.},
  journal={IEEE Trans. Inform. Theory},
  volume={64},
  number={1},
  pages={77-92},
  year={2018},
  month={January}
}

@Article{EPI:25arxiv,
  author =      "Gavalakis, L. and Kontoyiannis, I.",
  title={Conditions for equality and stability in {Shannon's and Tao's}
	entropy power inequalities},
  journal = {arXiv e-prints},
  volume = {\texttt{2509.14021 [math.PR]}},
  year =        "2025",
  month =       "September",
  pages =       "",
}

@Article{gavalakis-goh:arxiv,
  author =      "Gavalakis, L. and Goh, M.K. and Kontoyiannis, I.",
  title={Entropy lower bounds and sum-product phenomenay},
  journal = {arXiv e-prints},
  volume = {\texttt{2604.20233 [math.CO]}},
  year =        "2026",
  month =       "April",
  pages =       "",
}

@book{diaconis:book,
  title={Group representations in probability and statistics},
  author={Diaconis, P.},
  year={1988},
  publisher={Institute of Mathematical Statistics},
  address={Clevenald, OH},
}

@article{harremoes:09b,
  title={Maximum entropy on compact groups},
  author={Harremo{\"e}s, P.},
  journal={Entropy},
  volume={11},
  number={2},
  pages={222-237},
  year={2009},
}

@book{hewitt:bookI,
  title={Abstract harmonic analysis: Volume I: Structure of 
	topological groups, integration theory, group representations},
  author={Hewitt, E. and Ross, K.A.},
  year={1979},
  publisher={Springer},
  edition={Second},
  address={Berlin, Germany},
}

@book{hewitt:bookII,
  title={Abstract harmonic analysis: Volume II: Structure and analysis for 
	compact groups, analysis on locally compact abelian groups},
  author={Hewitt, E. and Ross, K.A},
  year={1970},
  publisher={Springer},
  address={Berlin, Germany},
}

@book{srivastava:book,
  title={A course on Borel sets},
  author={Srivastava, S.M.},
  year={1998},
  publisher={Springer},
  address={New York, NY}
}

@Book{pinsker:book,
  author =      "Pinsker, M.S.",
  title =       "Information and information stability
		 of random variables and processes",
  publisher =   "Holden-Day",
  year =        "1964",
  address =     "San Fransisco",
  OPTsummary =  ""
}

@book{liese-vajda:book,
  title={Convex statistical distances},
  author={Liese, F. and Vajda, I.},
  series={Teubner-Texte zur Mathematik},
  publisher={BSB BG Teubner Verlagsgesellschaft},
  volume={95},
  year={1987},
  address={Leipzig},
}

@article{KV:83,
     title = {Random Walks on Discrete Groups: {B}oundary and Entropy},
     author = {Kaimanovich, V.A. and Vershik, A.M.},
     journal = {Ann. Probab.},
     volume = {11},
     number = {3},
     pages = {457-490},
     year = {1983},
	MONTH = {August},
    }

@article{shamai:90,
  title={A binary analog to the entropy-power inequality},
  author={Shamai, S. and Wyner, A.D.},
	journal = {IEEE Trans. Inform. Theory},
  volume={36},
  number={6},
  pages={1428-1430},
  year={1990},
  month={November},
}

@article{eldan:20,
	author = {Eldan, R. and Mikulincer, D.},
	journal = {Probab. Theory Related Fields},
	number = {3-4},
	pages = {891-922},
	title = {Stability of the {S}hannon-{S}tam inequality via 
		the {F}\"{o}llmer process},
	volume = {177},
	year = {2020},
}

@article{tao:18,
  title={An inverse theorem for an inequality of {K}neser},
  author={Tao, T.},
  journal={Proc. Steklov Inst. Math.},
  volume={303},
  number={1},
  pages={193-219},
  year={2018},
  month={March},
}

@article{christ:22,
  title={Inequalities of {Riesz-Sobolev} type for compact connected 
	abelian groups},
  author={Christ, M. and Iliopoulou, M.},
  journal={Am. J. Math.},
  volume={144},
  number={5},
  pages={1367-1435},
  year={2022},
  month={October},
}

@article{major:79,
  title={A local limit theorem for the convolution of probability 
	measures on a compact connected group},
  author={Major, P. and Shlosman, S.B.},
  journal="Z. Wahrsch. Verw. Gebiete",
  volume={50},
  number={2},
  pages={137-148},
  year={1979},
}

@article{choi:94,
  title={Equivalence of certain entropy contraction coefficients},
  author={Choi, M.D. and Ruskai, M.B. and Seneta, E.},
  journal={Linear Algebra Appl.},
  volume={208},
  pages={29-36},
  year={1994},
}

@article{green:25,
  title={Sumsets and entropy revisited},
  author={Green, B. and Manners, F. and Tao, T.},
  journal={Random Struct. Algorithms},
  volume={66},
  number={1},
  pages={e21252},
  year={2025},
}

@article{blachar:arxiv,
  title={Small entropy doubling for random walks and polynomial growth},
  author={Blachar, G.},
  journal = {arXiv e-prints},
  volume = {\texttt{2604.00490 [math.GR]}},
  month = {April},
  year={2026},
}

@Article{tao:10,
  author =      "Tao, T.",
  title =       "Sumset and inverse sumset theory for {Shannon} entropy",
  journal =     "Comb. Probab. Comput.",
  year =        "2010",
  volume =      "19",
  number =      "4",
  pages =       "603-639",
  month =       "July",
  OPTnote =     ""
}

@article{wyner:73II,
  title={A theorem on the entropy of certain binary sequences 
	and applications: {Part II}},
  author={Wyner, A.D.},
  JOURNAL = {IEEE Trans. Inform. Theory},
  volume={19},
  number={6},
  pages={772-777},
  year={1973},
  month={November}
}

@article{wyner:73I,
  title={A theorem on the entropy of certain binary sequences 
	and applications: {Part I}},
  author={Wyner, A.D. and Ziv, J.},
  JOURNAL = {IEEE Trans. Inform. Theory},
  volume={19},
  number={6},
  pages={769-772},
  year={1973},
  month={November}
}

@article{csiszar:72,
  title={A class of measures of informativity of observation channels},
  author={Csisz{\'a}r, I.},
  journal={Period. Math. Hung.},
  volume={2},
  number={1-4},
  pages={191-213},
  year={1972},
}

@preamble{
   "\def\cprime{$'$} "
}

@article {donsker-varadhan:I-II,
    AUTHOR = {Donsker, M.D. and Varadhan, S.R.S.},
     TITLE = {Asymptotic evaluation of certain {M}arkov process expectations
              for large time. {I}. {I}{I}},
   JOURNAL = {Comm. Pure Appl. Math.},
    VOLUME = {28},
      YEAR = {1975},
     PAGES = {1-47; ibid. 28:279-301},
}

@article{varju:13,
  title={Random walks in compact groups},
  author={Varj{\'u}, P.P.},
  journal={Doc. Math.},
  volume={18},
  pages={1137-1175},
  year={2013}
}

@article{kneser:56,
  title={Summenmengen in lokalkompakten abelschen {Gruppen}},
  author={Kneser, M.},
  journal={Math. Z.},
  volume={66},
  number={1},
  pages={88-110},
  year={1956},
}

@article{rosenthal:94,
  title={Random rotations: {Characters} and random walks on {SO(N)}},
  author={Rosenthal, J.S.},
  journal =     "Ann. Probab.",
  pages={398-423},
  year={1994},
  month={January}
}

@article{ball:12,
author = {Ball, K. and Nguyen, V.H.},
journal = {Studia Mathematica},
number = {1},
pages = {81-96},
title = {Entropy jumps for isotropic log-concave random vectors and spectral gap},
volume = {213},
year = {2012},
}

@article{ball:03,
author = {Ball, K. and Barthe, F. and Naor, A.},
title = {{Entropy jumps in the presence of a spectral gap}},
volume = {119},
journal = {Duke Math. J.},
number = {1},
pages = {41-63},
year = {2003},
month={July}
}

@article {blachman:65,
    AUTHOR = {Blachman, N.M.},
     TITLE = {The convolution inequality for entropy powers},
   JOURNAL = {IEEE Trans. Inform. Theory},
    VOLUME = {11},
    NUMBER = {2},
     MONTH = {April},
      YEAR = {1965},
     PAGES = {267-271},
}

@article {stam:59,
    AUTHOR = {Stam, A.J.},
     TITLE = {Some inequalities satisfied by the quantities of information
              of {F}isher and {S}hannon},
   JOURNAL = {Inf. Contr.},
    VOLUME = {2},
      YEAR = {1959},
     PAGES = {101-112},
	NUMBER = {2},
}

@article {johnson:00b,
  title={Entropy and convergence on compact groups},
  author={Johnson, O. and Suhov, Yu.M.},
  journal =     {J. Theoret. Probab.},
  volume={13},
  number={3},
  pages={843-857},
  year={2000},
}

@InProceedings{polyanskiy-wu:17,
author="Polyanskiy, Y. and Wu, Y.",
editor="Carlen, E. and Madiman, M.  and Werner, E.M.",
title="Strong Data-Processing Inequalities for Channels 
	and {Bayesian} Networks",
booktitle="Convexity and Concentration",
year="2017",
publisher="Springer New York",
address="New York, NY",
pages="211-249",
}

@article{jog:14,
  title={The Entropy Power Inequality and {Mrs. Gerber's Lemma} for 
	Groups of Order {$2^n$}},
  author={Jog, V. and Anantharam, V.},
  journal =     "IEEE Trans. Inform. Theory",
  volume={60},
  number={7},
  pages={3773-3786},
  year={2014},
  month={July},
}

@article{madiman:21b,
  title={Entropy inequalities for sums in prime cyclic groups},
  author={Madiman, M. and Wang, L. and Woo, J.O.},
  journal={SIAM J.Discrete Math.},
  volume={35},
  number={3},
  pages={1628-1649},
  year={2021},
}

@article{abbe:14,
  title={A new entropy power inequality for integer-valued random variables},
  author={Haghighatshoar, S. and Abbe, E. and Telatar, I.E.},
  JOURNAL = "IEEE Trans. Inform. Theory",
  volume={60},
  number={7},
  pages={3787-3796},
  year={2014},
  month={July},
}

@article{courtade:18,
	author = {Courtade, T.A. and Fathi, M. and Pananjady, A.},
	journal = {IEEE Trans. Inform. Theory},
	number = {8},
	pages = {5691-5703},
	title = {Quantitative stability of the entropy power inequality},
	volume = {64},
	year = {2018},
	month = {August},
}

@article{kuratowski:65,
  title={A general theorem on selectors},
  author={Kuratowski, K. and Ryll-Nardzewski, C.},
  journal={Bull. Acad. Polon. Sci. S{\'e}r. Sci. Math. Astronom. Phys.},
  volume={13},
  number={6},
  pages={397-403},
  year={1965}
}

@Article{shannon:48,
	author= "Shannon, C.E.",
	title=  "A mathematical theory of communication",
	journal="Bell System Tech. J.",
	volume= "27",
   	number= "3",
	pages=  "379-423, 623-656",
	year=   "1948",
	OPTsummary=""
	}
